\documentclass[12pt]{amsart}
\usepackage[utf8]{inputenc}

\usepackage[unicode=true,pdfusetitle,
 bookmarks=true,bookmarksnumbered=false,bookmarksopen=false,
 breaklinks=false,pdfborder={0 0 0},pdfborderstyle={},backref=false,colorlinks=false]
 {hyperref}

\usepackage{graphicx}
\usepackage[usenames,dvipsnames]{color}

\usepackage{geometry}
\usepackage{amssymb,amsmath}
\usepackage{amsfonts}
\usepackage{amsaddr}

\numberwithin{equation}{section}
\numberwithin{figure}{section}

\numberwithin{equation}{section}
\numberwithin{figure}{section}

\theoremstyle{plain}
\newtheorem{thm}{Theorem}[section]
\newtheorem{cor}[thm]{Corollary}

\newtheorem{prop}[thm]{Proposition}

\newtheorem{ex}[thm]{Example}

\theoremstyle{definition}

\begin{document}

\title{Clifford spectrum of three 2 by 2 matrices  }
\author{Alexander Cerjan}
\email{awcerja@sandia.gov}
\address{Center for Integrated Nanotechnologies, Sandia National Laboratories, Albuquerque, New Mexico 87185, USA}

\author{Vasile Lauric}
\email{vasile.lauric@famu.edu}
\address{Department of Mathematics, Florida A\&M University, Tallahassee, Florida 32307, USA}

\author{Terry A. Loring}
\email{tloring@unm.edu}
\address{Department of Mathematics and Statistics, University of New Mexico, Albuquerque, New Mexico 87131, USA}

\date{}

\begin{abstract}
We prove that the Clifford spectrum associated to three 2 by 2 matrices is nonempty. The structure of Clifford is described in terms "moving" level curves.  We discuss some implication of a conjecture formulated for arbitrary size n by n of three matrices and give an example in the case of three self-adjoint operators in the infinite dimensional Hilbert space.
\end{abstract}

\maketitle

\tableofcontents

\section{Introduction}

Many forms of joint spectrum have been introduced for $n$-tuples of noncommuting Hermitian matrices, all producing the same result in the commutative case.  There are no algorithms to compute some of them.  Some are very often nonempty.  Others do not allow for a non-trivial spectral mapping theorem.  Here we will focus on the Clifford spectrum, which is getting used in the study of quantum materials but will first briefly review some of the other options.

To define the Weyl spectrum of $\mathbf{A}=(A_1,\dots,A_n)$ one first needs to define the Weyl functional calculus $f \mapsto \mathcal{W}_f(\mathbf{A})$. See \cite{anderson1969weylFucntCalc,anderson1967functional,jefferies2004Book_NC_spectral} for the essential definitions. In a simple case, such as $f(x,y)=x^2y$, the definition is
$\mathcal{W}_f(X,Y)) = \tfrac{1}{3}\left( X^2Y+XYX+YX^2 \right).$ One can define
\begin{equation*}
\mathcal{W}_{f}(\mathbf{A})=\left(2\pi\right)^{-n/2}\int_{\mathbb{R}^{n}}\hat{f}(\mathbf{y})\exp\left(-i\sum y_{j}A_{j}\right)\thinspace d\mathbf{y}
\end{equation*}
so long as $f$ is a Schwartz function.
One can extend this to a distribution using methods in functional analysis or Clifford analysis.  The Weyl spectrum $\sigma_{\textrm{W}}(\mathbf{A})$ is the support of this distribution.  This version of joint spectrum satisfies a spectral mapping theorem and, in particular, is always nonempty.

A different way to find a noncommutative spectrum is to generalize
the resolvent $\left(A-\lambda I\right)^{-1}$ and its association
with the spectrum give by
\begin{equation*}
\sigma(A)=\left\{ \left.\lambda\in\mathbb{R}\right|\left(A-\lambda I\right)^{-1}\text{ does not exist}\right\} .
\end{equation*}
This is for the case of one Hermitian matrix. This is equivalent to
\begin{equation*}
\sigma(A)=\left\{ \lambda\in\mathbb{R}\left|\left(A-\lambda I\right)^{2}\text{ is singular}\right.\right\} .
\end{equation*}
One generalization here is
\begin{equation*}
\Lambda^{\mathrm{Q}}(\mathbf{A})=\left\{ \boldsymbol{\lambda}\in\mathbb{R}^{n}\left|\sum_{j}\left(A_{j}-\lambda_{j}I\right)^{2}\text{ is singular}\right.\right\} ,
\end{equation*}
which we prefer to call the \emph{quadratic spectrum}. Thus our replacement
for the (inverse of) the resolvant in this setting is
\begin{equation*}
Q_{\boldsymbol{\lambda}}(\mathbf{A})=\sum_{j}\left(A_{j}-\lambda_{j}I\right)^{2}.
\end{equation*}
This can be easily estimated on a computer  and has a nice interpretation in terms of joint  eigenvalues \cite{cerjan_quadratic_2022}.  It is, however, empty in many interesting examples since a joint eigenvector implies some manner of commutativity in the matrices.  

The Clifford spectrum \cite{kisil1996mobius} is a variation on the quadratic spectrum.  It is still easy to estimate on a computer and is intimately tied to $K$-theory \cite{loringPseuspectra} as it related to topological physics.

Topological phases of matter in physical systems \cite{hasan_colloquium_2010,chiu_review_2016} can be detected through operator-theoretic constructions that package a system's Hamiltonian together with its position observables into a single composite operator called the spectral localizer \cite{loringPseuspectra,LoringSchuBa_odd,LoringSchuBa_even,Cerjan_2024}, that has emerged as a valuable tool for classifying topology in materials that cannot be considered using standard band-theoretic methods \cite{fulga_aperiodic_2016,cerjan_local_2022,wong_probing_2023,cheng_revealing_2023,cerjan_local_2024,chadha_real-space_2024,franca_topological_2024,qi_real-space_2024,ghosh_local_2024,ochkan_non-hermitian_2024,spataru_local_ag_2025}. In this setting, the Clifford spectrum of $d$ self-adjoint matrices $\mathbf{A}=(A_1, A_2,...)$, connected to a physical system in $d-1$ spatial dimensions, is the set of $\boldsymbol{\lambda} \in \mathbb{R}^{d}$ for which the spectral localizer
\begin{equation*}
    L_{\boldsymbol{\lambda}}(\mathbf{A}) = \sum_{j=1}^{d} \left(A_j - \lambda_j I\right) \otimes \Gamma_{j}
\end{equation*}
fails to be invertible, in which $\Gamma_{j}$ form a representation of a Clifford algebra of at least size $d$ \cite{cerjan_even_2024}. At the core of the spectral localizer framework are matrix homotopy arguments: all of the relevant invariants for classifying material topology (e.g., signature-, Pfaffian-, or determinant-based) can change only when a path of spectral localizers crosses a non-invertible point. Consequently, any topologically nontrivial material must force zeros of its Clifford spectrum, as the spectral localizer’s singularities are the necessary “gateways’’ through which the homotopy class can change. In turn, these zeros also signal the presence of nearby states \cite{cerjan_quadratic_2022}, thereby encoding bulk-boundary correspondence and furnishing quantitative robustness estimates for the associated local topological markers. It is therefore natural to ask for general structural guarantees, namely when is the Clifford spectrum nonempty?

The Clifford spectrum is easily seen to have many of other properties expected of a spectrum of matrices; it is a compact subset of $\mathbb{R}^n$, respects direct sums and is left fixed when the matrices are conjugated by a unitary, for example.  It extends to cover noncommuting tuples of Hermitian elements in a $C^*$-algebra \cite{cerjan2025multivariable_Clifford}, and the Clifford spectrum behaves as expected with respect to $*$-homomorphisms.  Note that the spectral mapping theorem fails dramatically for the Clifford spectrum \cite{cerjan2025multivariable_Clifford}.  It seems we need to study many forms of spectrum for noncommuting Hermitian matrices depending on what properties we want that spectrum to have. 

For many forms joint spectrum for noncommuting matrices \cite{jefferies2004Book_NC_spectral,pryde1993joint_spectra} we find that we do not know that the  joint spectrum must always be nonempty, and indeed sometimes this is false \cite{cerjan_quadratic_2022}.   The Clifford spectrum, where $L_{\boldsymbol{\lambda}}(\mathbf{A})$ is not just close to singular, but is actually singular, it critical for the applications to physics as these singularities correspond to changes in local topology in a material.  The many calculations done by physicists using the spectral localizer amount to ample evidence that the Clifford spectrum is never empty.

Kisil \cite{kisil1996mobius} claimed to have shown the Clifford spectrum is alway nonempty by showing that the related monogenic spectrum is nonempty.  It was subsequently shown \cite{jefferies1998weyl} that these spetra are not always equal, so the conjecture on the Clifford spectrum is still open.  The Clifford spectrum is the only joint spectrum known to have applications to topological protection in materials so this conjecture seems important.

Here, we give a complete answer in the smallest nontrivial case: for any three $2 \times 2$ Hermitian matrices, the Clifford spectrum is never empty. After reducing by unitary conjugacy so that $A_1$ is diagonal and expanding the localizer in the Pauli basis, we derive a closed formula for $\det[L_{\boldsymbol{\lambda}}(\mathbf{A})]$ that is a quartic polynomial $D(\boldsymbol{\lambda})$ with an explicit symmetry center and then show that $D(\boldsymbol{\lambda}) \le 0$ for some $\boldsymbol{\lambda} \in \mathbb{R}^{3}$. Equivalently, $L_{\boldsymbol{\lambda}}(\mathbf{A})$ is singular somewhere for every triple $\mathbf{A}$, establishing non-emptiness of their Clifford spectrum.

The contrast between these three froms of nomcommutative joint spectrum is already clear in the special case of three 2-by-2 Hermitian matrices $(A_1,A_2,A_3)$.  Greiner and Ricker \cite{Greiner1993joint_spectra_commutativity} proved that the quadratic spectrum is empty unless these three matrices all commute.

In section 2 we introduce some notation and recall what the Clifford spectrum is associated to three self-adjoint matrices. The expression of the determinant of the localizer is used to show that the Clifford spectrum is nonempty. The section 3 is devoted to establishing the formula of the determinant.

A special case that illustates how different these spectral can be is $\mathbf{A}=(\tfrac{1}{2}\sigma_x, \sigma_y, \tfrac{1}{2}\sigma_z)$.  Here the quadratic spectrum is empty, the Weyl spectrum is a elipsoid \cite{anderson1967functional}, while the Clifford spectrum is not even a manifold, but is a rotated lemniscate of Bernoulli \cite[\S 4]{debonis2022surfaces_Cliff_spec}.

\section{Three two-by-two hermitian matrices.}

We will denote by $\sigma_0,\ \sigma_1,\ \sigma_2,$ and $\sigma_3$ the following 2 by 2 matrices
  \begin{equation}
    \label{eqn:two-by-two-projections}
      \left[\begin{array}{cc}
                   1 & 0\\
                   0 & 1
     \end{array}\right],\ 
    \left[\begin{array}{cc}
                  0 & 1\\
                  1 & 0
     \end{array}\right],\
     \left[\begin{array}{cc}
                   0 & -i\\
                   i & 0
     \end{array}\right],\
     \left[\begin{array}{cc}
                    1 & 0\\
                    0 & -1
     \end{array}\right],\
\end{equation} respectively.
Any self-adjoint matrix in $M_2(\Bbb{C})$  can be written as a linear combination of $\sigma_0,\dots, \sigma_3$ with real coefficients. Let $A_j,\ j=1,2,3$ be three self-adjoint matrices written as $$A_j=\sum_{k=0}^3  a_{jk} \sigma_k.$$ 
Since Clifford spectrum is invariant under unitary conjugacy, we may assume that $A_1$ is a diagonal matrix, that is $a_{11}=a_{12}=0,$ that is $A_1$ is a linear combination of only $\sigma_0$ and $\sigma_3.$
Recall that the localizer operator is defined as
$L_{\boldsymbol{x}}(\Bbb{A})=\sum_{j=1}^3  (A_j-x_j \sigma_0)\otimes \Gamma_j,$ where $\Gamma_j$'s form a Clifford system and we will assume in our calculations that $\Gamma_j=\sigma_j,\ j=1,2,3$, and $\boldsymbol{x}=(x_1,x_2,x_3)\in\Bbb{R}^3.$
Thus the localizer operator
\begin{align*}
   L_{\boldsymbol{x}}(\Bbb{A})=
   & (a_{10}-x_1)\, \sigma_0\otimes \sigma_1+
                               a_{13}\, \sigma_3 \otimes \sigma_1+\\
   & (a_{20}-x_2)\, \sigma_0\otimes \sigma_2+a_{21}\, \sigma_1\otimes \sigma_2+
            a_{22}\, \sigma_2\otimes \sigma_2+a_{23}\, \sigma_3\otimes \sigma_2+\\
   & (a_{30}-x_3)\, \sigma_0\otimes \sigma_3+a_{31}\, \sigma_1\otimes \sigma_3+
             a_{32}\, \sigma_2\otimes \sigma_3+a_{33}\, \sigma_3\otimes \sigma_3, 
\end{align*}
or
$$L_x(\Bbb{A})=\sum_{j=1}^3 \sum_{k=0}^3 b_{jk} \sigma_{kj},
$$
where $b_{j0}=a_{j0}-x_j$ for $j=1,2,3,$\  $b_{jk}=a_{jk}$ for $j,k=1,2,3$,  $b_{11}=b_{12}=0,$
and $\sigma_{kj}=\sigma_k \otimes \sigma_j,$
and the Clifford spectrum, denoted $\Lambda^\mathrm{C}(\Bbb{A})$ is defined
\begin{align*}
    \Lambda^\mathrm{C}(\Bbb{A})=\{{x}\in\Bbb{R}^3: \, L_{x}(\Bbb{A}) \text{ is not invertible}\}.
\end{align*}

Next proposition provides the equation of the determinant, $D(x),$ of the localizer operator $L_{\boldsymbol{x}}(\Bbb{A}).$ With the above notation and assumption that matrix $A_1$ is diagonal, we have the following.
  
\begin{prop} \begin{align*}
     D(x)&=\sum_{j=1}^3 ((a_{j0}-x_j)-a_{j3})^2 \sum_{j=1}^3 ((a_{j0}-x_j)+a_{j3})^2 +\\
         &+2 |\alpha_2|^2[((a_{10}-x_1)^2-a_{13}^2)-((a_{20}-x_2)^2-a_{23}^2)+
                ((a_{30}-x_3)^2-a_{33}^2)]+\\
         &+2 |\alpha_3|^2[((a_{10}-x_1)^2-a_{13}^2)+((a_{20}-x_2)^2-a_{23}^2)-
               ((a_{30}-x_3)^2-a_{33}^2)]-\\
         &-4\, Im(\alpha_2 \, \alpha_3)\,
              [((a_{20}-x_2)+a_{23})((a_{30}-x_3)-a_{33})+
              ((a_{20}-x_2)-a_{23})((a_{30}-x_3)+a_{33})]+\\
        &+|\alpha_2^2-\overline{\alpha_3}^2|^2,
\end{align*}
  where $\alpha_2=a_{22}+i\, a_{21}$ and $\alpha_3=a_{31}+i\, a_{32}.$
\end{prop}

The proof of this proposition will be given in section 3 of this note.
The following is an obvious consequence of the Proposition 2.1 and  will be used later.

\begin{cor} 
     $D(x_1,x_2,x_3)= D(2\,a_{10}-x_1, 2\,a_{20}-x_2,2\,a_{30}-x_3),$ 
     and consequently any level surface $D(\boldsymbol{x})=t$ is symmetric with 
     respect to the point $(a_{10},a_{20},a_{30}),$ 
     and in particular the Clifford spectrum $\Lambda^C(\Bbb{A}).$ 
\end{cor}

Thus we may assume that $a_{j0}=0,\, j=1,2,3,$ that is the matrices $A_j$ 
can be assumed to be linear combinations of only $\sigma_1,\, \sigma_2$ and \
$\sigma_3,$ (recall that $A_1$ does not include $\sigma_1$ and $\sigma_2$ 
since it is assumed to be diagonal.) This reduction can be also obtained by observing
that the Clifford spectrum "shifts" when the triplet $\Bbb{A}$ is perturbed by
$\boldsymbol{y} I,$ that is $\Lambda^\mathrm{C}(\Bbb{A}+{\boldsymbol{y} I})=\Lambda^\mathrm{C}(\Bbb{A})+{\boldsymbol{y}},$
and thus each matrix of $\Bbb{A}$ is of trace zero.

We denote by $a$  the vector $ (a_{13}, a_{23}, a_{33})$ and by $||\cdot||$ the Euclidean norm 
of a vector in $\Bbb{R}^3$ and $|\alpha|^2:=|\alpha_2 |^2+|\alpha_3 |^2,$ where $\alpha_2$ and $\alpha_3$ are the complex numbers in Proposition 2.1. Furthermore, we denote by $c$
the following  expression 
\begin{equation}
    c:=||a||^4-2\, |\alpha|^2\, ||a||^2+4\, |\alpha_2 \,a_{23}+i\overline{\alpha_3}\, a_{33}|^2 +| \alpha_2^2-\overline{\alpha_3}^2 |^2
\end{equation}    
and by $A$ the 3 by 3 matrix
\begin{equation}
    \label{eqn:three-by-three}
      A:=\left[\begin{array}{ccc}
                   a_{13}^2 & a_{13} a_{23} & a_{13} a_{33}\\
                   a_{13} a_{23} & a_{23}^2+|\alpha_2|^2 & a_{23} a_{33}+\operatorname{Im}(\alpha_2  \alpha_3)\\
                   a_{13} a_{33} &  a_{23} a_{33}+\operatorname{Im}(\alpha_2  \alpha_3) & a_{33}^2+|\alpha_3|^2 
     \end{array}\right].\ 
    \end{equation}
With the assumption that $a_{j0}=0,\ j=1,2,3$ and the notation used above, one can verify (elementary calculation) the following.

\begin{prop}
\begin{align*}
     D({x})=||x||^4 +
      2\, ||x||^2 \big(||a||^2 +|\alpha|^2\big)-4\,\langle Ax,x \rangle +c.
\end{align*}
\end{prop}

\begin{thm}
For any three self-adjoint matrices in $M_2(\Bbb{C}),$ the Clifford spectrum is nonempty.  
\end{thm}
\begin{proof}
It is enough to prove that there exists $x\in\Bbb{R}^3$ so that $D(x)\le 0.$ 
If $c\le 0,$ then $x=0$ will satisfy the inequality. 
Thus, in what follows we will assume that $c>0.$ Since $x=0$ is not a solution for the inequality $D(x)\le0,$ we can assume that $x\ne 0.$ Thus, for
$x\ne 0$
we can write
\begin{align*}
     D(x)=||x||^4 +
      2\, ||x||^2 \big(||a||^2 +|\alpha|^2-2\,\langle A\frac{x}{||x||},\frac{x}{||x||} \rangle\big) +c.
\end{align*}
Observe that $D(x)$ cannot be less or equal to zero unless
$$E(x):=||a||^2+|\alpha|^2-2\,\langle A\frac{x}{||x||},\frac{x}{||x||} \rangle$$ 
is negative for some $x\in \Bbb{R}^3\setminus\{0\}.$ 
Thus the problem is reduced  to finding
$\sup_{||x||=1} \langle Ax, x\rangle$
since such a point will provide the lowest value of $E(x).$
Denote 
     $$\lambda_0=\sup_{||x||=1} \langle Ax, x\rangle$$  
and let $x_0$, $||x_0||=1,$ so that $\langle A x_0,x_0 \rangle=\lambda_0.$
We will show that for $x_*$ defined by 
    $$x_*=\sqrt{2\,\lambda_0 -(||a||^2+|\alpha|^2)}\,x_0,$$
we have $D(x_*)=-(2\,\lambda_0 -(||a||^2+|\alpha|^2))^2+c\le 0,$ or equivalently 
\begin{equation}
    (2\,\lambda_0 -(||a||^2+|\alpha|^2))^2\ge c.
\end{equation}
We will justify later  that $2\,\lambda_0\ge ||a||^2+|\alpha|^2.$
Denote $||a||^2+|\alpha|^2$ by $\beta$, and thus inequality (2.4) is equivalent to
$$4\,\lambda_0^2-4\,\lambda_0\beta+\beta^2-c\ge 0.$$
Denote $|\alpha_2 \,a_{23}+i\,\overline{\alpha_3}\, a_{33}|$ by $\delta,$
and thus 
\begin{align*}
    \beta^2-c&=||a||^4+2\,||a||^2 |\alpha|^2+|\alpha|^4-(||a||^4-2\,||a||^2 
                 |\alpha|^2+4\,\delta^2+| \alpha_2^2-\overline{\alpha_3}^2 |^2)=\\
              &=4\,||a||^2 |\alpha|^2-4\,\delta^2+|\alpha|^4-| \alpha_2^2-\overline{\alpha_3}^2 |^2.
\end{align*}
First we look at
\begin{align*}
    |\alpha|^4-| \alpha_2^2-\overline{\alpha_3}^2 |^2&=(|\alpha_2|^4+2\,|\alpha_2 \,\alpha_3|^2+|\alpha_3|^4)-(|\alpha_2|^4-2\,Re((\alpha_2\,\alpha_3)^2)+|\alpha_3|^4)=\\
    &= 2\,|\alpha_2 \,\alpha_3|^2+2\,Re((\alpha_2\,\alpha_3)^2)=4\,(Re(\alpha_2\,\alpha_3))^2,
\end{align*}
and thus $\beta^2-c=4\,||a||^2 |\alpha|^2-4\,\delta^2+4\,(Re(\alpha_2\,\alpha_3))^2,$
and consequently inequality (2.4) is rewritten as
\begin{equation}
    \lambda_0^2-\beta\,\lambda_0+d\ge 0,
\end{equation}
where $d=||a||^2\,|\alpha|^2-\delta^2+(Re(\alpha_2\,\alpha_3))^2.$

On other hand, $\lambda_0$ is the largest eigenvalue of matrix $A$ and it satisfies
$det(A-\lambda_0\,  I_3)=0.$ Next we will expand $det(A-\lambda_0\, I_3).$
Recall 
\begin{equation*}
    \label{eqn:three-by-three}
      A-\lambda_0\,I_3:=\left[\begin{array}{ccc}
                   a_{13}^2-\lambda_0 & a_{13} a_{23} & a_{13} a_{33}\\
                   a_{13} a_{23} & a_{23}^2+|\alpha_2|^2-\lambda_0 & a_{23} a_{33}+\operatorname{Im}(\alpha_2  \alpha_3)\\
                   a_{13} a_{33} &  a_{23} a_{33}+\operatorname{Im}(\alpha_2  \alpha_3) & a_{33}^2+|\alpha_3|^2-\lambda_0 
     \end{array}\right],\ 
    \end{equation*}
and thus
\begin{align*}
    det(A-\lambda_0 I_3)&=(a_{13}^2-\lambda_0)(a_{23}^2+|\alpha_2|^2-\lambda_0)(a_{33}^2+|\alpha_3|^2-\lambda_0)+\\
    &+2\,(a_{13}^2\, a_{23}\, a_{33})(a_{23} a_{33}+\operatorname{Im}(\alpha_2  \alpha_3))-
    (a_{13} a_{33})(a_{23}^2+|\alpha_2|^2-\lambda_0)-\\
    &-(a_{13} a_{23})^2 (a_{33}^2+|\alpha_3|^2-\lambda_0)-( a_{23} a_{33}+\operatorname{Im}(\alpha_2  \alpha_3) )^2(a_{13}^2-\lambda_0)=\\
    &=-\lambda_0^3+\lambda_0^2\,\big[a_{13}^2+(a_{23}^2+|\alpha_2|^2)+(a_{33}^2+|\alpha_3|^2)\big]-\\
    &-\lambda_0\,\big[ (a_{23}^2+|\alpha_2|^2)(a_{33}^2+|\alpha_3|^2)+a_{13}^2\,(a_{23}^2+|\alpha_2|^2)+a_{13}^2\, (a_{33}^2+|\alpha_3|^2) -\\
    &-(a_{13}\,a_{23})^2 -(a_{13}\,a_{33})^2 -(a_{23} a_{33}+\operatorname{Im}(\alpha_2  \alpha_3))^2\big]
    +det(A).
\end{align*}
Next, it only remains to verify that the above expression of $det(A-\lambda_0\, I_3)$
is equivalent to
\begin{equation}
    det(A-\lambda_0 I_3)=-\lambda_0^3+\beta\,\lambda_0^2-d\,\lambda_0+det(A).
\end{equation}

First, we need to justify that $2\,\lambda_0\ge \beta.$ Indeed, let
$f(t)=t^3-\beta\, t^2+d\, t-det(A)$ and observe that  $\beta^2<3d$ implies
$f^{'}(t)\ge 0,$ and thus the equation $f(t)=0$ would have only one real solution 
and implicitelly all eigenvalues of matrix $A$ would be identical, which is obviously not the case.
Furthermore, $\lambda_0\ge \lambda_+,$ where $\lambda_+$ is the bigger zero of $f^{'}(t),$ that is
$\lambda_+=\frac{1}{2}(\beta+\sqrt{\beta^2-3d}),$ and consequently $\lambda_0\ge\frac{\beta}{2}.$

We turn now to the final step. Since $det(A-\lambda_0)=0$ and $\lambda_0>0,$ equality (2.6) implies
$\lambda_0^2-\beta\,\lambda_0+d=\frac{det(A)}{\lambda_0}.$
But 
$det(A)=a_{13}^2\, \big[|\alpha_2\, \alpha_3|^2-(\operatorname{Im}(\alpha_2 \,\alpha_3))^2\big]
=a_{13}^2\,(Re(\alpha_2\,\alpha_3))^2\ge0, $ and thus $\frac{det(A)}{\lambda_0}\ge 0$
and that completes the proof of the theorem.
\end{proof}

\section{Proof of Proposition 2.1}

In this section we will provide the proof of Proposition 2.1. The proof is elementary, but contains cumbersome calculations; the determinant of a 4 by 4 matrix contains $4!$ products, and since each term of each product is a linear combination of either 2 or 4 other terms, the expression of the determinant is indeed long. Fortunately, after careful observations one  can group some of the terms and the expression becomes reasonable.

Recall 
$$
    L_x(\Bbb{A})=\sum_{j=1}^3  y_{j0}\, \sigma_{0j} +      \sum_{j=1}^3 \sum_{k=1}^3 a_{jk} \sigma_{kj},
$$
where 
$y_{j0}=a_{j0}-x_j$ for $j=1,2,3,$\    $a_{11}=a_{12}=0,$
 $\sigma_{kj}=\sigma_k \otimes \sigma_j,$ and $\alpha_2=a_{22}+i\,a_{21},\ \alpha_3=a_{31}+i\,a_{32}.$
 Furthermore, denote $y_{j0}\pm a_{j3}$ by ${y_{j3}}_{\pm}$ and thus 
\begin{equation}
 \label{eqn:four-by-four}
    L_x(\Bbb{A})=\left[\begin{array}{cccc}
      {y_{{30}_+}}  & {y_{{10}_+}} -i\,{y_{{20}_+}}  & \overline{\alpha_3} & -\alpha_2\\
       {y_{{10}_+}} +i\,{y_{{20}_+}} & -{y_{{30}_+}} & \alpha_2 &-\overline{\alpha_3}\\
       \alpha_3 & \overline{\alpha_2} & {y_{30_-}}  & {y_{10_-}} -i\,{y_{20_-}} \\
       -\overline{\alpha_2} & -\alpha_3 &  {y_{10_-}} +i\,{y_{20_-}} & -{y_{30_-}}
     \end{array}\right].\ 
    \end{equation}
The determinant contains 24 products of which 12 will be associated with the even permutations 
$\sigma_1=id,$ $\sigma_2=(1,2)(3,4),$ $\sigma_3=(1,3)(2,4),$ $\sigma_4=(1,4)(2,3),$ $\sigma_5=(1,2,3),$
$\sigma_6=(1,3,2),$ $\sigma_7=(1,2,4),$ $\sigma_8=(1,4,2),$ $\sigma_9=(1,3,4),$ $\sigma_{10}=(1,4,3),$
$\sigma_{11}=(2,3,4),$ and $\sigma_{12}=(2,4,3)$ and other 12 associated with the odd permutations
$\tau_1=(1,2),$ $\tau_2=(1,3),$ $\tau_3=(1,4),$ $\tau_4=(2,3),$ $\tau_5=(2,4),$ $\tau_6=(3,4),$
$\tau_7=(1,2,3,4),$ $\tau_8=(1,2,4,3),$ $\tau_9=(1,3,2,4),$ $\tau_{10}=(1,3,4,2),$ $\tau_{11}=(1,4,2,3),$
and $\tau_{12}=(1,4,3,2).$

We will denote by $P_{\gamma}$ the product including the sign corresponding to each permutation $\gamma\in S_4.$
Thus,

\begin{align*}
    P_{\sigma_1}=(y_{30}+a_{33})[-(y_{30}+a_{33})](y_{30}-a_{33})[-(y_{30}-a_{33})]=
     (y_{30}+a_{33})^2(y_{30}-a_{33})^2=(y_{30}^2-a_{33}^2)^2;
\end{align*}

\begin{align*}
    P_{\sigma_2}&=[(y_{10}+a_{13})+i\,(y_{20}+a_{23})]\,[(y_{10}+a_{13})-i\,(y_{20}+a_{23})]\cdot\\
    &\cdot [(y_{10}-a_{13})+i\,(y_{20}-a_{23})]\,[(y_{10}-a_{13})-i\,(y_{20}-a_{23})]=\\
    &=[(y_{10}+a_{13})^2+(y_{20}+a_{23})^2][(y_{10}-a_{13})^2+(y_{20}-a_{23})^2]=\\
    &=(y_{10}^2-a_{13}^2)^2+(y_{20}^2-a_{23}^2)^2+(y_{10}-a_{13})^2\,(y_{20}+a_{23})^2+(y_{10}+a_{13})^2\,(y_{20}-a_{23})^2;
\end{align*} 

\begin{align*}
    P_{\sigma_3}=|\alpha_3|^4;\ \  P_{\sigma_4}=|\alpha_2|^4;
\end{align*}

\begin{align*}
    P_{\sigma_5}&=[(y_{10}+a_{13})+i\,(y_{20}+a_{23})]\,
    (\overline{\alpha_2}\, \overline{\alpha_3})\,[-(y_{30}-a_{33})]=\\
    &=-\overline{\alpha_2 \alpha_3}[ (y_{10}+a_{13})(y_{30}-a_{33})+i\,(y_{20}+a_{23})(y_{30}-a_{33})  ];
\end{align*}

\begin{align*}
    P_{\sigma_6}&=\alpha_2\, \alpha_3\, [(y_{10}+a_{13})-i\,(y_{20}+a_{23})]\,[-(y_{30}-a_{33})]=\\
    &=-{\alpha_2 \alpha_3}[ (y_{10}+a_{13})(y_{30}-a_{33})-i\,(y_{20}+a_{23})(y_{30}-a_{33})  ];
\end{align*}

\begin{align*}
    P_{\sigma_7}&=\alpha_2\,\alpha_3\, [(y_{10}+a_{13})+i\,(y_{20}+a_{23})]\,(y_{30}-a_{33})=\\
    &=\alpha_2 \alpha_3\,[ (y_{10}+a_{13})(y_{30}-a_{33})+i\,(y_{20}+a_{23})(y_{30}-a_{33})  ];
\end{align*}

\begin{align*}
    P_{\sigma_8}&=(\overline{\alpha_2}\, \overline{\alpha_3})\,
    [(y_{10}+a_{13})-i\,(y_{20}+a_{23})]\,(y_{30}-a_{33})=\\
    &=\overline{\alpha_2 \alpha_3}[ (y_{10}+a_{13})(y_{30}-a_{33})-i\,(y_{20}+a_{23})(y_{30}-a_{33})  ];
\end{align*}

\begin{align*}
    P_{\sigma_9}&={(-\alpha_2)}\, ({\alpha_3})\,
    [(y_{10}-a_{13})+i\,(y_{20}-a_{23})]\,[-(y_{30}+a_{33})]=\\
    &={\alpha_2 \alpha_3}[ (y_{10}-a_{13})(y_{30}+a_{33})+i\,(y_{20}-a_{23})(y_{30}+a_{33})  ];
\end{align*}

\begin{align*}
    P_{\sigma_{10}}&={-(\overline{\alpha_2})}\, (\overline{\alpha_3})\,
    [(y_{10}-a_{13})-i\,(y_{20}-a_{23})]\,[-(y_{30}+a_{33})]=\\
    &=\overline{\alpha_2 \alpha_3}\,[ (y_{10}-a_{13})(y_{30}+a_{33})-i\,(y_{20}-a_{23})(y_{30}+a_{33})  ];
\end{align*}

\begin{align*}
    P_{\sigma_{11}}&={(\overline{\alpha_2})}\, -(\overline{\alpha_3})\,
    [(y_{10}-a_{13})+i\,(y_{20}-a_{23})]\,[(y_{30}+a_{33})]=\\
    &=-(\overline{\alpha_2 \alpha_3})\,[ (y_{10}-a_{13})(y_{30}+a_{33})+i\,(y_{20}-a_{23})(y_{30}+a_{33})  ];
\end{align*}

\begin{align*}
    P_{\sigma_{12}}&={(\overline{\alpha_2})}\, [-(\overline{\alpha_3})]\,
    [(y_{10}-a_{13})+i\,(y_{20}-a_{23})]\,[(y_{30}+a_{33})]=\\
    &=-(\overline{\alpha_2 \alpha_3})\,[ (y_{10}-a_{13})(y_{30}+a_{33})+i\,(y_{20}-a_{23})(y_{30}+a_{33})  ].
\end{align*}
One can observe some similarity in the products as follows,
\begin{align*}
     P_{\sigma_{6}}+ P_{\sigma_{7}}=2i\,(\alpha_2 \alpha_3)\, (y_{20}+a_{23})(y_{30}-a_{33})
\end{align*}
and
\begin{align*}
     P_{\sigma_{5}}+ P_{\sigma_{8}}=-2i\,(\overline{\alpha_2 \alpha_3})\, (y_{20}+a_{23})(y_{30}-a_{33})
\end{align*}
and thus
\begin{align*}
     P_{\sigma_{5}}+ P_{\sigma_{6}}+P_{\sigma_{7}}+ P_{\sigma_{8}}
     &=2(i\,\alpha_2 \alpha_3-i\,\overline{\alpha_2\alpha_3})\, (y_{20}+a_{23})(y_{30}-a_{33})=\\
     &=(-4)\,\operatorname{Im}(\alpha_2 \alpha_3)\,(y_{20}+a_{23})(y_{30}-a_{33}).
\end{align*}
Similarly
\begin{align*}
     P_{\sigma_{10}}+ P_{\sigma_{11}}=-2i\,(\overline{\alpha_2 \alpha_3})\, (y_{20}-a_{23})(y_{30}+a_{33})
\end{align*}
and
\begin{align*}
     P_{\sigma_{9}}+ P_{\sigma_{12}}=2i\,({\alpha_2 \alpha_3})\, (y_{20}-a_{23})(y_{30}+a_{33}),
\end{align*}
and therefore
\begin{align*}
     P_{\sigma_{9}}+ P_{\sigma_{10}}+P_{\sigma_{11}}+ P_{\sigma_{12}}
     &=2(i\,\alpha_2 \alpha_3-i\,\overline{\alpha_2\alpha_3})\, (y_{20}+a_{23})(y_{30}-a_{33})=\\
     &=(-4)\,\operatorname{Im}(\alpha_2 \alpha_3)\,(y_{20}-a_{23})(y_{30}+a_{33}).
\end{align*}
Consequently,
\begin{align*}
    S_e:=\sum_{k=1}^{12} P_{\sigma_k}&=\sum_{j=1}^3 (y_{j0}^2-a_{j3}^2)^2+
    (y_{10}+a_{13})^2\,(y_{20}-a_{23})^2+(y_{10}-a_{13})^2\,(y_{20}+a_{23})^2+\\
    &+|\alpha_2|^4 +|\alpha_3|^4-4\, \operatorname{Im}(\alpha_2 \alpha_3)\,[(y_{20}+a_{23})(y_{30}-a_{33})+(y_{20}-a_{23})(y_{30}+a_{33})].
\end{align*}
We now calculate the products associated with odd permutations, including the sign of the permutation, thus
\begin{align*}
    P_{\tau_1}&=(-1)[(y_{10}+a_{13})+i(y_{20}+a_{23})]\,[(y_{10}+a_{13})-i(y_{20}+a_{23})]
      (y_{30}-a_{33})[-(y_{30}-a_{33})]=\\
     &=(y_{30}-a_{33})^2\,[(y_{10}+a_{13})^2+(y_{20}+a_{23})^2];
\end{align*}
\begin{align*}
    P_{\tau_2}&=(-1)(\alpha_3)(\overline{\alpha_3})\,[-(y_{30}+a_{33})][-(y_{30}-a_{33})]=-|\alpha_3|^2\, (y_{30}^2-a_{33}^2);
\end{align*}
\begin{align*}
    P_{\tau_3}&=(-1)(-\alpha_2)(-\overline{\alpha_2})\,[-(y_{30}+a_{33})][(y_{30}-a_{33})]=|\alpha_2|^2\, (y_{30}^2-a_{33}^2);
\end{align*}
\begin{align*}
    P_{\tau_4}&=(-1)(\alpha_2)(\overline{\alpha_2})\,(y_{30}+a_{33})[-(y_{30}-a_{33})]=|\alpha_2|^2\, (y_{30}^2-a_{33}^2);
\end{align*}
\begin{align*}
    P_{\tau_5}&=(-1)(-\alpha_3)(-\overline{\alpha_3})\,(y_{30}+a_{33})(y_{30}-a_{33})=-|\alpha_3|^2\, (y_{30}^2-a_{33}^2);
\end{align*}
\begin{align*}
    P_{\tau_6}&=(-1)[(y_{10}-a_{13})+i(y_{20}-a_{23})]\,[(y_{10}-a_{13})-i(y_{20}-a_{23})]
      (y_{30}+a_{33})[-(y_{30}+a_{33})]=\\
     &=(y_{30}+a_{33})^2\,[(y_{10}-a_{13})^2+(y_{20}-a_{23})^2];
\end{align*}
\begin{align*}
    P_{\tau_7}&=(-1)(-\alpha_2)(\overline{\alpha_2})\,[(y_{10}+a_{13})+i(y_{20}+a_{23})]
    [(y_{10}-a_{13})+i(y_{20}-a_{23})]=\\
    &=|\alpha_2|^2\, [(y_{10}+a_{13})(y_{10}-a_{13})-(y_{20}+a_{23})(y_{20}-a_{23})+\\
    &+i\, (y_{10}+a_{13})(y_{20}-a_{23})+i\,(y_{10}-a_{13})(y_{20}+a_{23});
\end{align*}
\begin{align*}
    P_{\tau_8}&=(-1)(-\alpha_3)(\overline{\alpha_3})\,[(y_{10}+a_{13})+i(y_{20}+a_{23})]
    [(y_{10}-a_{13})-i(y_{20}-a_{23})]=\\
    &=|\alpha_3|^2\, [(y_{10}+a_{13})(y_{10}-a_{13})+(y_{20}+a_{23})(y_{20}-a_{23})-\\
    &-i\, (y_{10}+a_{13})(y_{20}-a_{23})+i\,(y_{10}-a_{13})(y_{20}+a_{23});
\end{align*}
\begin{align*}
    P_{\tau_9}&=(-1)(-\alpha_2)({\alpha_2})(\alpha_3)(-\alpha_3)=-(\alpha_2 \alpha_3)^2;
\end{align*}
\begin{align*}
    P_{\tau_{10}}&=(-1)(-\alpha_3)(-\overline{\alpha_3})\,[(y_{10}-a_{13})+i(y_{20}-a_{23})]
    [(y_{10}+a_{13})-i(y_{20}+a_{23})]=\\
    &=|\alpha_3|^2\, [(y_{10}+a_{13})(y_{10}-a_{13})+(y_{20}+a_{23})(y_{20}-a_{23})+\\
    &+i\, (y_{10}+a_{13})(y_{20}-a_{23})-i\,(y_{10}-a_{13})(y_{20}+a_{23});
\end{align*}
\begin{align*}
    P_{\tau_{11}}&=(-1)(-\overline{\alpha_2})(\overline{\alpha_2})(-\overline{\alpha_3})(\overline{\alpha_3})=-(\overline{\alpha_2 \alpha_3})^2;
\end{align*} 
\begin{align*}
    P_{\tau_{12}}&=(-1)(\alpha_2)(-\overline{\alpha_2})\,[(y_{10}-a_{13})-i(y_{20}-a_{23})]
    [(y_{10}+a_{13})-i(y_{20}+a_{23})]=\\
    &=|\alpha_2|^2\, [(y_{10}+a_{13})(y_{10}-a_{13})-(y_{20}+a_{23})(y_{20}-a_{23})+\\
    &-i\, (y_{10}+a_{13})(y_{20}-a_{23})-i\,(y_{10}-a_{13})(y_{20}+a_{23}).
\end{align*}
As seen in the case of even permutations, also the products associated with odd permutations share some similarities, thus
\begin{align*}
    P_{\tau_2}+P_{\tau_3}+P_{\tau_4}+P_{\tau_5}=
    2\,(|\alpha_2|^2-|\alpha_3|^2)\, (y_{30}+a_{33})(y_{30}-a_{33}),
\end{align*}
\begin{align*}
    P_{\tau_1}+P_{\tau_6}&=(y_{30}-a_{33})^2\, (y_{10}+a_{13})^2+
    (y_{30}-a_{33})^2\, (y_{20}+a_{23})^2+\\
    &+(y_{10}-a_{13})^2\, (y_{30}+a_{33})^2+
    (y_{20}-a_{23})^2\, (y_{30}+a_{33})^2,
\end{align*}
\begin{align*}
     P_{\tau_9}+P_{\tau_{11}}=-(\alpha_2 \alpha_3)^2-(\overline{\alpha_2 \alpha_3})^2,
\end{align*}  
\begin{align*}
    P_{\tau_7}+P_{\tau_8}+P_{\tau_{10}}+P_{\tau_{12}}&=
            |\alpha_2|^2\, (2\, y_{10+} y_{10-}-2\, y_{20+} y_{20-} +\\
    &+i\, y_{10+} y_{20-}+i\, y_{10-} y_{20+}
            -i\, y_{10-} y_{20+}-i\, y_{10+} y_{20-})\\
    &+|\alpha_3|^2\, (2\, y_{10+} y_{10-}+2\, y_{20+} y_{20-} -\\
    &-i\, y_{10+} y_{20-}+i\, y_{10-} y_{20+}
            +i\, y_{10+} y_{20-}-i\, y_{10-} y_{20+})=\\
    &=2\,(|\alpha_2|^2+|\alpha_3|^2)\, y_{10+}y_{10-}+ 2\,(|\alpha_3|^2-|\alpha_2|^2)\, y_{20+}y_{20-},
\end{align*}
where we used the notation 
$y_{j0+}=y_{j0}+a_{j3}$ and $y_{j0-}=y_{j0}-a_{j3},$ for $j=1,2,3.$
Thus,
\begin{align*}
    \sum_{j=2,3,4,5,7,8,10,12} P_{\tau_j}=2\,[(|\alpha_2|^2+|\alpha_3|^2)\, (y_{10+}y_{10-})
    +(|\alpha_3|^2-|\alpha_2|^2)\, (y_{20+}y_{20-})+(|\alpha_2|^2-|\alpha_3|^2)\, (y_{30+}y_{30-})],
\end{align*}
and consequently
\begin{align*}
    \sum_{j=1}^{12} P_{\tau_j}&=(y_{10+})^2 (y_{30-})^2+(y_{20+})^2 (y_{30-})^2+(y_{10-})^2 
             (y_{30+})^2+(y_{20-})^2 (y_{30+})^2-2\,Re((\alpha_2\alpha_3)^2)+\\
             &+2\,[(|\alpha_2|^2+|\alpha_3|^2)\, (y_{10+}y_{10-})
    +(|\alpha_3|^2-|\alpha_2|^2)\, (y_{20+}y_{20-})+(|\alpha_2|^2-|\alpha_3|^2)\, (y_{30+}y_{30-})].
\end{align*}
Finally,
\begin{align*}
    D(x)&=\sum_{j=1}^3 (y_{j0+} y_{j0-})^2+\sum_{j\ne k}^3 (y_{j0+} y_{k0-})^2+\\
    &+|\alpha_2|^4+|\alpha_3|^4-2\, Re((\alpha_2 \alpha_3)^2)-\\
    &-4\, \operatorname{Im}(\alpha_2 \alpha_3)\, (y_{20+} y_{30-}+y_{20-} y_{30+} ) +\\
    &+2\,[(|\alpha_2|^2+|\alpha_3|^2)\, (y_{10+}y_{10-})
    +(|\alpha_3|^2-|\alpha_2|^2)\, (y_{20+}y_{20-})+(|\alpha_2|^2-|\alpha_3|^2)\, (y_{30+}y_{30-})].
\end{align*}
Observe now the following equalities:
\begin{align*}
    \sum_{j=1}^3 (y_{j0+} y_{j0-})^2+\sum_{j\ne k}^3 (y_{j0+} y_{k0-})^2= 
    \sum_{j=1}^3 (y_{j0+})^2  \sum_{j=1}^3 (y_{j0-})^2 ,
\end{align*}
\begin{align*}
    |\alpha_2|^4+|\alpha_3|^4-2\, Re((\alpha_2 \alpha_3)^2)=|\alpha_2^2-\overline{\alpha_3}^2|^2,
\end{align*}
and
\begin{align*}
    &2\,[(|\alpha_2|^2+|\alpha_3|^2)\, (y_{10+}y_{10-})
    +(|\alpha_3|^2-|\alpha_2|^2)\, (y_{20+}y_{20-})+
    (|\alpha_2|^2-|\alpha_3|^2)\, (y_{30+}y_{30-})]=\\
    & 2\, |\alpha_2|^2\, (y_{10+} y_{10-}-y_{20+} y_{20-}+y_{30+} y_{30-})+
      2\, |\alpha_3|^2\, (y_{10+} y_{10-}+y_{20+} y_{20-}-y_{30+} y_{30-}).
\end{align*}
Replacing $y_{j0\pm}$ by $(x_{j0}-a_{j0})\pm a_{j3}$ respectively, one obtains 
the expression for $D(x)$ that is stated in Proposition 2.1.

\section{Geometrical structure of the Clifford spectrum}

In this section we will show that the Clifford spectrum
is related to the classical Cassini 2d-ovals (some two-dimensional surface) 
and that in some particular cases it degenerates into either two points or one point. 
It is thus adequate to name such surface a {\it Cassini 2dd-oval} (two-dimensional deformed oval). 
Recall that the Cassini 1d-ovals are the curves defined as the 
loci of points in the plane so that the product of the distances 
to two fixed points (say, $F_1$ and $F_2$ called foci) is a constant $r.$
Depending on the value of $r,$ the loci can consist of either the points 
$F_1$ and $F_2,$ or two disjoint loops around the foci,  or
the loops can touch each other (like the $\infty$ symbol), or just 
one loop containing inside both foci 
(we will call these the c1, c2, c3, c4 phases respectively). 
The Cassini 1d-ovals are symmetric with respect to the midpoint of 
the segment joining the foci, as well as, to the line that joins the foci.  
Obviously, when $F_1=F_2,$ all these curves become either a single point or a circle. 
We will call a {\it classical Cassini 2d-oval} a surface that is obtained 
by rotating a Cassini 1d-oval around the line that joins the two foci. 
When the foci are confounding, such a surface is either a sphere
or simply it degenerates to a point.

The expression of $D(x),$ after $a_{j0}$'s are set to zero, becomes
\begin{align*}
     D(x)&=\big(\sum_{j=1}^3 (x_j+a_{j3})^2 \big)\big(\sum_{j=1}^3 (x_j-a_{j3})^2 \big) +\\
         &+2 |\alpha_2|^2[(x_1^2-a_{13}^2)-(x_2^2-a_{23}^2)+
                (x_3^2-a_{33}^2)]+\\
         &+2 |\alpha_3|^2[(x_1^2-a_{13}^2)+(x_2^2-a_{23}^2)-
               (x_3^2-a_{33}^2)]-\\
         &-4\, \operatorname{Im}(\alpha_2 \, \alpha_3)\,
              \big[(x_2-a_{23})(x_3+a_{33})+
              (x_2+a_{23})(x_3-a_{33})\big]+\\
        &+|\alpha_2^2-\overline{\alpha_3}^2|^2.
\end{align*}

Observe that the first product in the expression of $D(x)$ is $(d(P,F_+))^2 (d(P,F_-))^2,$ 
where $d(P,F_+)$ and $d(P,F_-)$ are the distances from 
$P(x_1,x_2,x_3)$ to $F_{\pm}(\pm a_{13},\pm a_{23},\pm a_{33}),$ respectively.
Let us denote by $H(x)$ the following expression:
\begin{align*}
        H(x)= &-2 |\alpha_2|^2[(x_1^2-a_{13}^2)-(x_2^2-a_{23}^2)+(x_3^2-a_{33}^2)]\\
         &-2 |\alpha_3|^2[(x_1^2-a_{13}^2)+(x_2^2-a_{23}^2)- (x_3^2-a_{33}^2)]-\\
         &+4\, \operatorname{Im}(\alpha_2 \, \alpha_3)\,
              \big[(x_2-a_{23})(x_3+a_{33})+
              (x_2+a_{23})(x_3-a_{33})\big]
        -|\alpha_2^2-\overline{\alpha_3}^2|^2,
\end{align*}
thus $D(x)=\big(d(P,F_+)\,(d(P,F_-)\big)^2-H(x).$
The graph of the equation $D(x)=0$ can be viewed as a union of
curves that are obtained by intersecting the graph of equation
$$\big(d(P(x),F_+)\, d(P(x),F_-)\big)^2=r$$ (classical Cassini 2d-oval)
with the graph of $H(x)=r$ (a hyperboloid) for $r\ge 0.$
As $r$ varies in the interval $[0,+\infty)$, the intersection between 
the classical Cassini 2d-ovals and the hyperboloids
will eventually (large values of $r$) be empty since the 
Cassini surfaces "grow" at a slower pace 
(passing through all phases from c1 to c4) than the hyperboloids.
We need to justify that the graph of the quadratic equation $H(x)=r$ is indeed a hyperboloid,
passing through phases h1 (hyperboloid with one sheet), h2 (cone), and h3 (hyperboloid with two sheets) 
as $r$ increases from $0$ to $+\infty$. We first rewrite the expression of $H(x)$
\begin{align*}
        H(x)= &-2 |\alpha_2|^2[( x_1^2-x_2^2+x_3^2)-(a_{13}^2-a_{23}^2+a_{33}^2)]-\\
              &-2 |\alpha_3|^2[( x_1^2+x_2^2-x_3^2)-(a_{13}^2+a_{23}^2-a_{33}^2)]-\\
              &+8\, \operatorname{Im}(\alpha_2 \, \alpha_3)\,
              (x_2 x_3-a_{23} a_{33})
        -|\alpha_2^2-\overline{\alpha_3}^2|^2=\\
        =&-2 x_1^2 (|\alpha_2|^2+|\alpha_3|^2)+2 x_2^2 (|\alpha_2|^2-|\alpha_3|^2)+2 x_3^2 (|\alpha_3|^2-|\alpha_2|^2)+\\
        &+8\, \operatorname{Im}(\alpha_2 \, \alpha_3)\,(x_2 x_3)+e,
\end{align*}
where
$$
e=2 |\alpha_2|^2 (a_{13}^2-a_{23}^2+a_{33}^2)+2 |\alpha_3|^2 (a_{13}^2+a_{23}^2-a_{33}^2)-
        8\, \operatorname{Im}(\alpha_2 \, \alpha_3)\,(a_{23} a_{33})-|\alpha_2^2-\overline{\alpha_3}^2|^2.
$$
Selecting the following transformation
$$x_1=y_1,\ x_2=\cos(\theta) y_2+\sin(\theta) y_3,\ x_3=-\sin(\theta) y_2+\cos(\theta) y_3$$
(that is a rotation that does not change axis $x_1$ and only $x_2,\ x_3$ are rotated
in their plane by an angle $\theta$), the expression of $H(x)$ becomes
\begin{align*}
        G(y)= &-2 |\alpha_2|^2 \big( y_1^2-\cos(2\theta)\, y_2^2-2 \sin(2\theta)\, y_2 y_3+\cos(2\theta)\, y_3^2\big)-\\
              &-2 |\alpha_3|^2 \big( y_1^2+\cos(2\theta)\, y_2^2+2 \sin(2\theta)\, y_2 y_3-\cos(2\theta)\, y_3^2 \big)-\\
              &+8\, \operatorname{Im}(\alpha_2 \, \alpha_3)\,\big(-\frac{1}{2}\,\sin(2\theta)\, y_2^2 +\cos(2\theta)\, y_2 y_3 +\frac{1}{2}\, \sin(2\theta)\, y_3^2  \big)+e=\\
            = &-2 \big(|\alpha_2|^2+|\alpha_3|^2\big)\, y_1^2 +\\
              &+2 \big(|\alpha_2|^2 \cos(2\theta)-|\alpha_3|^2 \cos(2\theta)-2 \operatorname{Im}(\alpha_2 \alpha_3)\, \sin(2\theta)  \big)\,y_2^2+\\
              &+2 \big(-|\alpha_2|^2 \cos(2\theta)+|\alpha_3|^2 \cos(2\theta)+2 \operatorname{Im}(\alpha_2 \alpha_3)\, \sin(2\theta)  \big)\,y_3^2+\\
              &+4 \big( |\alpha_2|^2 \sin(2\theta)-|\alpha_3|^2 \sin(2\theta)+
              2 \operatorname{Im}(\alpha_2 \alpha_3)\, \cos(2\theta)     \big)\, y_2 y_3+e.
\end{align*}
We select $\theta$ so that the coefficient of $y_2 y_3$ is zero, that is
$$(|\alpha_2|^2-|\alpha_3|^2) \sin(2\theta)+2 \operatorname{Im}(\alpha_2 \alpha_3)\cos(2\theta)=0.$$

Let us denote by $f$ the following 
$$f:=\sqrt{\big(|\alpha_2|^2-|\alpha_3|^2   \big)^2+4 \big(\operatorname{Im}(\alpha_2 \alpha_3)\big)^2}.$$
We first consider the case when $f=0.$
Thus $f=0$ is equivalent to $|\alpha_2|=|\alpha_3|=:t$ and $\operatorname{Im}(\alpha_2 \alpha_3)=0,$
that is, $\alpha_2=t e^{i\beta}$ and $\alpha_3=\pm t e^{-i\beta}.$ The sub-case of $t=0$ is equivalent
to matrices $A_2$ and $A_3$ being diagonal (recall that $A_1$ is assumed to be diagonal),
and thus the Clifford spectrum in this case consists of two points, $F_+$ and $F_-,$ 
or just one point if $F_+=F_-.$
If $t>0,$ then $e=4 t^2 a_{13}^2,$ and the equation  $D(x)=0$ is equivalent to
$\big(d(P(x),F_+) \,d(P(x),F_-)  \big)^2=4t^2(a_{13}^2-x_1^2).$ 
In this case, $F_+$ and $F_-$ are satisfying the previous equation. We remark here that if $a_{13}=0,$
that is matrix $A_1=0_2,$ then $F_+$ and $F_-$ are the only points in the Clifford spectrum located in the plane $x_1=0$,
of course with possibility of them coinciding. 

We turn now to the case $f>0.$
There exists a unique $\theta_0\in [0, 2\pi)$ so that 
$$\frac{|\alpha_2|^2-|\alpha_3|^3}{f}=\cos(\theta_0)$$
and
$$\frac{2 \operatorname{Im}(\alpha_2 \alpha_3)}{f}=\sin(\theta_0),$$
and select $\theta=-\frac{\theta_0}{2}\in(-\pi,0],$ that is $\sin(\theta_0+2\theta)=0.$
Consequently the coefficient of $y_2^2$ is $1$ or $-1$ and the coefficient of $y_3^2$ is $-1$ or $1,$ respectively.
We will only discuss the case when the coefficient of $y_2^2$ is $1$ and the coefficient 
of $y_3^2$ is $-1;$ the other case is similar and will be omitted.
Thus $G(y)=-2(|\alpha_2|^2+|\alpha_3|^2)\, y_1^2 +2f\,y_2^2-2f\,y_3^2+e=r\ge 0$
is equivalent to 
$$
\frac{y_1^2}{A}+\frac{y_2^2}{B}-\frac{y_3^2}{B}=1,
$$
where
$A={s\over{-2(|\alpha_2|^2+|\alpha_3|^2)}},$  $B={s\over{2f}},$ and $s=r-e\in [-e,+\infty).$
We distinguish three situations, namely (i) $e>0,$ (ii)  $e=0$, and (iii) $e<0.$ 
Consequently, as variable $r$ varies from $0$ to $+\infty$ (thus $s$ varies from $-e$ to $+\infty$) in case (i)
the shapes of graphs of $G(y)=r$ vary from a hyperboloid with one sheet, advancing to a cone, and then to a hyperboloid with two sheets, in case (ii) from cone to hyperboloid with two sheets, and in case (iii) only hyperboloid with two sheets.

Observe that $c=||a||^4-e$ (see equation 2.2), and thus the three cases (i)-(iii) are equivalent to 
$||a||^4>c,$ $||a||^4=c,$ and $||a||^4<c,$ respectively. Recall that if $c<0$ or $c=0$ (see the proof of Theorem 2.4), then $D(0)<0$ or $D(0)=0,$ respectively 
and in this cases the Clifford spectrum is a connected set (one bubble, or two touching bubbles).

Here are some examples corresponding $c<0$ and $c=0,$ respectively.
\begin{ex} \label{ex:1}
  Let $a=(2,1,2)$ and $\alpha_2=1-3i$ and $\alpha_3=2+i;$ thus $|\alpha_2|^2=10,$
  $|\alpha_3|^2=5,$ $\alpha_2^2-{\overline{\alpha_3}}^2=-11-2i$, $|\alpha_2^2-{\overline{\alpha_3}}^2|^2=125,$ 
  $\operatorname{Im}(\alpha_2 \alpha_3)=-5,$ $||a||^4=81$
  and consequently $e=105$ and $c=81-105=-24$. The equation of $D(x)$ is
  $$\big[(x_1-2)^2+(x_2-1)^2+(x_3-2)^2  \big]\big[(x_1+2)^2+(x_2+1)^2+(x_3+2)^2   \big]+
  30x_1^2-10x_2^2+10 x_3^2-105.$$
\end{ex}

\begin{figure*}[t]
\center
\includegraphics[width=\columnwidth]{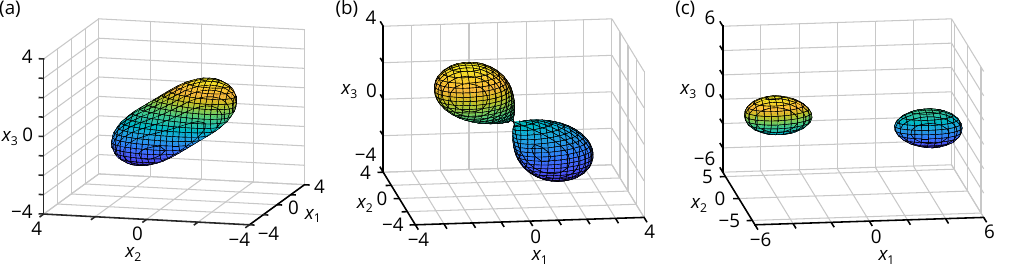}
\caption{Surfaces corresponding to the Clifford spectrum for the matrix triplets given in Example \ref{ex:1} (a), Example \ref{ex:2} (b), and Example \ref{ex:3} (c).
}
\label{fig:fig1}
\end{figure*}

\begin{ex} \label{ex:2}
  Let $a=(\sqrt{10},2,-2)$ and $\alpha_2=2-i$ and $\alpha_3=3+2i;$ thus $||a||^2=18,$ $|\alpha_2|^2=5,$
  $|\alpha_3|^2=13,$ $\alpha_2^2-{\overline{\alpha_3}}^2=-2+8i$, $|\alpha_2^2-{\overline{\alpha_3}}^2|^2=68,$ 
  $\operatorname{Im}(\alpha_2 \alpha_3)=1$
  and consequently $e=324,$ that is $||a||^4=e$. The equation of $D(x)$ is
  $$
    \big[(x_1-\sqrt{10})^2+(x_2-2)^2+(x_3+2)^2  \big]
              \big[(x_1+\sqrt{10})^2+(x_2+2)^2+(x_3-2)^2   \big]+
                   36x_1^2+16x_2^2-16 x_3^2-324$$
\end{ex}

The only case that the proof of Theorem 2.4 is nontrivial is $c>0,$ or equivalently $||a||^4>e.$
Observe that the starting point of the Cassini 2d-ovals, that is $r=0,$ the foci $F_+$ and $F_-$ are always in the "middle" connected component of the complement of the two-sheets hyperboloid. As $r$ increases from $0$ to $+\infty,$ and the two-sheets
hyperboloid is running away from the origin of the system, the Clifford spectrum contains two components (either two bubbles or two points). Here is an example.
\begin{ex} \label{ex:3}
  Let $a=(6,1,-1)$ and $\alpha_2=2-i$ and $\alpha_3=3+2i;$ thus $||a||^2=18,$ $|\alpha_2|^2=5,$
  $|\alpha_3|^2=13,$ $\alpha_2^2-{\overline{\alpha_3}}^2=-2+8i$, $|\alpha_2^2-{\overline{\alpha_3}}^2|^2=68,$ 
  $\operatorname{Im}(\alpha_2 \alpha_3)=1$
  and consequently $e=1236,$ and $||a||^4=1384$. The equation of $D(x)$ is
    $$\big[(x_1-6)^2+(x_2-1)^2+(x_3+1)^2  \big]
              \big[(x_1+6)^2+(x_2+1)^2+(x_3-1)^2   \big]+
                   36x_1^2+16x_2^2-16 x_3^2-1236.$$
\end{ex}

To this end, we display two examples that show that the Clifford spectrum can be two points different from the "foci",
or one point (which necessarily has to be the origin).

\begin{ex}
  Let $a=(1,1,0)$ and $\alpha_2=0$ and $\alpha_3=1+i;$ thus $||a||^2=2,$ $|\alpha_2|^2=0,$
  $|\alpha_3|^2=2,$  $|\alpha_2^2-{\overline{\alpha_3}}^2|^2=4,$ 
  $\operatorname{Im}(\alpha_2 \alpha_3)=0$
  and consequently $e=4=||a||^4$, thus $c=0.$ The equation of $D(x)$ is
 \begin{align*}
     &\big[(x_1-1)^2+(x_2-1)^2+x_3^2  \big]
              \big[(x_1+1)^2+(x_2+1)^2+x_3^2   \big]+
                   4x_1^2+4x_2^2-4x_3^2-4=\\
     &=\sum_{j=1}^3 x_j^4 +2\,\sum_{j<k }x_j^2 x_k^2 +4(x_1-x_2)^2
 \end{align*}
 which vanishes only when $x_1=x_2=x_3=0.$
    
\end{ex}

\begin{ex}
  Let $a=(1,1,0)$ and $\alpha_2=0$ and $\alpha_3=i;$ thus $||a||^2=2,$ $|\alpha_2|^2=0,$
  $|\alpha_3|^2=1,$  $|\alpha_2^2-{\overline{\alpha_3}}^2|^2=1,$ 
  $\operatorname{Im}(\alpha_2 \alpha_3)=0$
  and consequently $e=3$ and $||a||^4=4$, thus $c=1.$ The equation of $D(x)$ is
 \begin{align*}
     & \big[(x_1-1)^2+(x_2-1)^2+x_3^2  \big]
              \big[(x_1+1)^2+(x_2+1)^2+x_3^2   \big]+
                   2[(x_1^2-1)+(x_2^2-1)-x_3^2]+1=\\
     &=(x_1^2+x_2^2+x_3^2-1)^2  +4(x_1-x_2)^2 +4x_3^2
 \end{align*}
 which vanishes only when $x_1=x_2= \pm \frac{1}{\sqrt{2}}$ and $x_3=0.$
\end{ex}
 
\section*{Acknowledgments}

The middle author wants to dedicate this note to the memory of his mother, Georgeta.

A.C.\ and V.L.\ acknowledge support from the Laboratory Directed Research and Development program at Sandia National Laboratories. 
T.A.L.\ acknowledges support from the National Science Foundation, Grant No. DMS-2349959. 
This work was performed in part at the Center for Integrated Nanotechnologies, an Office of Science User Facility operated for the U.S. Department of Energy (DOE) Office of Science.
Sandia National Laboratories is a multimission laboratory managed and operated by National Technology \& Engineering Solutions of Sandia, LLC, a wholly owned subsidiary of Honeywell International, Inc., for the U.S. DOE's National Nuclear Security Administration under Contract No. DE-NA-0003525. 
The views expressed in the article do not necessarily represent the views of the U.S. DOE or the United States Government.

\bibliographystyle{amsplain}

\end{document}